\documentclass[11pt,fleqn]{amsart} 

\usepackage{amsthm}
\usepackage{amsfonts}
\usepackage[english]{babel}
\usepackage{graphicx}
\usepackage{soul}
\usepackage{stfloats}
\usepackage{morefloats}
\usepackage{cite}
\usepackage{lscape}
\usepackage{epstopdf}
\usepackage{braket}
\usepackage[lite]{amsrefs}
\usepackage{mathbbol}
\usepackage{tikz,tikz-cd}
\usepackage{shuffle}

\usepackage{algorithm,algorithmicx,algpseudocode}

\usepackage{amsmath,amstext,amsopn,amsfonts,eucal,amssymb}
\usepackage{graphicx,wrapfig,url}

\newcommand\N{{\mathbb N}}

\newcommand\R{{\mathbb R}}
\newcommand\C{{\mathbb C}}

\newtheorem{theorem}{Theorem}[section]

\newtheorem{lemma}[theorem]{Lemma}

\newtheorem{definition}[theorem]{Definition}

\DeclareMathOperator{\supp}{\rm supp}

\begin{document}

	\title[Approximation in TVS]{Universal approximation of continuous maps between topological vector spaces}

	\author{Emanuele Zappala} 
	\address{Department of Mathematics and Statistics, Idaho State University\\
		Physical Science Complex |  921 S. 8th Ave., Stop 8085 | Pocatello, ID 83209} 
	\email{emanuelezappala@isu.edu}

	\maketitle

	\begin{abstract}
			
			We study the problem of universal approximation of continuous maps between topological vector spaces by neural networks. We provide a universal approximation theorem for maps between locally convex topological vector spaces with and without paracompactness assumptions. We extend this result to continuous maps between non-locally convex topological vector spaces on compact sets under the assumption of finite topological dimension. 
	\end{abstract}

	\date{\empty}

	\tableofcontents

	\section{Introduction}
	
	Operator learning is a branch of machine learning that deals with the approximation of continuous maps between function spaces \cite{Lu,Lanthaler,Kovachki}, usually Banach or Hilbert spaces. One of the earliest universal approximation results for continuous maps between some general classes of Banach spaces is found in \cite{Chen-Chen}. 
	More recently, theoretical work on the approximation properties of neural networks for maps with either domain or codomain in topological vector spaces (TVSs) has been initiated in  \cite{Ismailov,Saini}. Part of the motivation for the interest in neural networks applied to more general TVSs than normed spaces, concerns the fact that in various setups, deep learning is involved with spaces of test functions and distributions, which naturally carry a locally convex TVS structure, but are not normed spaces. 
	
	The present work can be considered as an extension of the work of \cite{Ismailov,Saini}, where we consider also the case of non-locally convex TVSs. We point out, however some important differences that exist also between the approach followed in this article, and the results in \cite{Saini}. Unlike dual-based approaches of \cite{Saini}, our construction avoids continuous linear functionals on the ambient input space, relying instead on a finite-dimensional reduction built from points, neighborhoods, and partitions of unity. This provides an intrinsic finite-dimensional reduction that avoids reliance on continuous duals, which may be difficult to characterize explicitly in general TVSs. 
	
	The results in \cite{Bil-Xan} contain universal approximation results on vector lattices, based on the use of dual spaces. It might be of interest to investigate a variation of the current work applied to  vector lattices as in \cite{Bil-Xan}. 
	
	The main results of the present work show that continuous maps between locally convex TVSs can be approximated (over a compact) with arbitrary precision. See Theorem~\ref{thm:Universal} and Theorem~\ref{thm:Universal_no-paracompact}. When the underlying spaces are not locally convex, a similar result can be obtained under more restricting assumptions regarding the topological dimensionality of the compact.  
	
	\section{Preliminaries}
	
	In this section we recall some basic definitions, along with some notation and conventions used throughout the article. The results in this section are standard, and can be found for example in \cite{Rudin,Kantorovich-Akilov}. 
	
	\begin{definition}
		{\rm 
				A topological vector space (TVS) over a topological field $\mathbb k$ is a vector space endowed with a topology that makes the scalar multiplication and addition operations continuous. 
		}
	\end{definition}
	In this article, we focus our attention on the case where $\mathbb k = \R$, and all TVSs are Hausdorff. The latter condition is a common restriction, since if $X$ is a non-Hausdorff TVS, the quotient space $X/\overline{\{0\}}$ obtained by quotienting out the closure of the singleton containing the zero element of $X$ is a Hausdorff TVS. Over the reals (or complex numbers), a subset $A$ of a TVS is said to be \emph{balanced}, or \emph{circled},  if $\alpha A \subseteq A$ for every $|\alpha| \leq 1$. 
	
	We recall that every TVS has a fundamental system of balanced neighborhoods of $0$, where a fundametnal system of $0$ is a set which is a local basis at $0$. We will denote such system by $\mathcal U$. A TVS $X$ is said to be \emph{locally convex} when we can find a fundamental system $\mathcal U$ consisting of balanced and convex sets. The neighborhoods can be chosen open without loss of generality, and we will make this assumption for convenience. 

	Given a vector space $X$ over $\R$ or $\C$, a seminorm is a function $p : X \longrightarrow \R_+$ satisfying the properties: 
	\begin{itemize}
		\item 
			$p(x + y) \leq p(x) + p(y)$ for every $x,y \in X$; 
		\item 
			$p(sx) = |s|p(x)$ for every $x\in X$ and $t\in  \R$ or $\C$.
	\end{itemize} 
	If $X$ is a TVS endowed with a seminorm $p$, we define the ball centered at $x\in X$ of radius $r>0$ associated to the seminorm as $B_p(x,r) := \{y\in X\ |\ p(x-y) < r \}$, similarly to how balls are defined in a normed space. We recall the fundamental result that a TVS is locally convex if and only if its topology is generated by a family of seminorms $\{p_\lambda\}_{\lambda \in \Lambda}$, where $\Lambda$ is some indexing set. In this case, a neighborhood basis at $0$ is given by the family of sets of type $\bigcap_{j=1}^n B_{\lambda_j}(0,r_j)$, for some $n\in \N$ and $\lambda_1, \ldots, \lambda_n \in \Lambda$, where $B_{\lambda_j}(0,r_j)$ indicates the ball of radius $r_j$ centered at $0$ in the seminorm $p_{\lambda_j}$.  
	
	Let $X$ denote a topological space. The \emph{topological dimension}, also called \emph{Lebesgue covering dimension}, of $X$ is defined as the minimal natural number $n\in \N$, if it exists, such that every open covering of $X$ has a refinement with order at most equal to $n+1$. When a finite $n$ cannot be found, the topological dimension is said to be infinite. 
	
	In this article, we will use multiple times Kat\v{e}tov's Theorem~3 from \cite{Katetov}. This article was found to contain an error in a lemma that was used in the proof of the theorem as well, see \cite{Katetov2}. However, the correction article \cite{Katetov2} contains another approach that bypasses the problem, and therefore Theorem~3 in \cite{Katetov} stands as stated. We will refer to Theorem~3 of \cite{Katetov} throughout the present work, although for a proof one needs to consult also \cite{Katetov2}. The articles \cite{Frolik,Rob-Rob,Gar-Mil} contain results that can be used to obtain similar applications, although under stronger conditions than what needed in this work.

	\section{Universal approximation in locally convex topological vector spaces}
	
	In this section we consider the approximation of maps between locally convex topological vector spaces via neural networks. The main universal approximation theorem is first obtained under the assumption that the spaces are paracompact, and it is afterward shown how this hypothesis can be removed by constructing partitions of unity subordinate to given finite coverings of compacts. The approach followed in the proofs adapts earlier work in \cite{Projection} for Banach spaces. We mention some conceptual differences of interest before proceeding with the proofs. 
	
	Since the closed convex hull of a compact in a Banach space is compact, given a map $T : X \longrightarrow Y$ between two Banach spaces, where $K \subset X$ is compact, we can consider the approximation of $T$ on the closed convex hull of $K$. Therefore, it is enough to restrict one's attention to the case where a given compact is also convex. If $T : K \longrightarrow Y$ is continuous, we can continuously extend $T$ to $X$, and then again consider the closed convex hull of $K$. However, the closed convex hull of a compact is guaranteed to be compact under the assumption that $X$ be complete, and therefore downgrading $X$ from being a Banach space to being just a normed space would substantially change the aforementioned conclusions. 
	
	While an initial proof of the universal approximation result in \cite{Projection} made use of convexity, a more recent revised version has avoided the need for such assumption. As a consequence, the approaches of the proofs of Theorem~2.1 and Theorem~2.2 in \cite{Projection} remain valid also for normed spaces without any completeness requirement. In the present work, we consider a further adaptation of such approach to avoid the use of norms, and provide a version of the universal approximation result for locally convex TVSs. 
	
	Lastly, we mention that the problem of obtaining universal approximations for continuous maps $T : K \subset X \longrightarrow Y$, where $X$ and $Y$ are Banach (or simply normed) spaces and $K$ is compact, follows from the problem of universal approximations of continuous maps $T : X \longrightarrow Y$ over a compact $K \subset X$, since a map $T$ defined over $K$ can automatically be continuously extended to the whole space $X$. When $X$ and $Y$ are TVSs, we are not aware of analogous extension results that can be applied directly. However, as it will be shown in the first part of the proof below, if $T$ is continuous and defined on the compact $K$, we can approximate it with arbitrary precision on $K$ by a (uniformly) continuous map applying Kat\v{e}tov's Theorem 3 in \cite{Katetov}. Therefore, the universal approximation result below also gives an analogous result for maps of type $T : K \subset X \longrightarrow Y$. 
	
	\begin{theorem}\label{thm:Universal}
		Let $X$ and $Y$ be locally convex paracompact TVSs, let $T: X\longrightarrow Y$ be a continuous map, and let $K\subset X$ be a compact subset. Then, for any choice of an neighborhood $U$ of $0$, there exist natural numbers $n,m\in \mathbb N$, finite dimensional subspaces $E_n \subset X$ and $E_m \subset Y$, a continuous map $P_X : X\longrightarrow E_n$, and a neural network $f_{n,m} : \mathbb R^n \longrightarrow \mathbb R^m$ such that for every $x\in K$
		\begin{eqnarray}
		T(x) - \varphi_m^{-1}f_{n,m}\varphi_nP_X(x) &\in& U,
		\end{eqnarray}
		where $\varphi_k : E_k \longrightarrow \mathbb R^k$ indicates a linear homeomorphism between the finite dimensional space $E_k$ and $\mathbb R^k$.  
	\end{theorem}
	\begin{proof}
		We first show that we can approximate $T$ by a uniformly continuous map with arbitrary precision over $K$. Effectively, this would mean that we can assume without loss of generality that $T$ is uniformly continuous. 
		
		Given an open set $U$, we let $W\in \mathcal U$ be a convex open set such that $W \subset U$. We define the covering of $K$ given by $x + V_x$, for each $x\in K$, where $V_x \in \mathcal U$ is chosen in such a way that $y \in V_x$ implies $T(x) - T(y) \in W$. Obtain a finite subcovering $U_i := x_i + V_i$, with $i = 1, \ldots, r$, where $V_i$ is a shorthand notation for $V_{x_i}$. Let $\{\eta_i\}$ be a partition of unity subordinate to the covering $\{U_i\}_{i=1}^r$. Define the map $F : K \longrightarrow Y$ with the assignment 
		\begin{eqnarray*}
			F(x) &=& \sum_{i=1}^r \eta_i(x)T(x_i).  
		\end{eqnarray*}
		By the Heine-Cantor theorem, each function $\eta_i$ is uniformly continuous. Applying Theorem 3 in \cite{Katetov}, since TVSs are uniform spaces, each function $\eta_i$ can be extended to the whole space $X$. We indicate such extension by $\hat \eta_i$. The corresponding function 
		\begin{eqnarray*}
			\hat F(x) &=& \sum_{i=1}^n \hat \eta_i(x)T(x_i),   
		\end{eqnarray*}
		is uniformly continuous and extends $F$. By construction, it follows that $\hat F$ approximates the function $T$ over $K$ within $U$-precision. In fact, for $x\in K$, we have that
		\begin{equation*}
		T(x) - \hat F(x) = T(x) - F(x) = T(x) - \sum_{i=1}^n \eta_i(x)T(x_i) = \sum_{i=1}^n \eta_i(x)(T(x) - T(x_i)) \in W \subset U, 
		\end{equation*}
		by convexity of $W$. It therefore suffices to prove the theorem under the additional
		assumption that $T$ is uniformly continuous. Indeed, given the
		original neighborhood $U$, choose a balanced convex neighborhood
		$U_0$ such that $U_0+U_0\subseteq U$. The preceding construction
		provides a uniformly continuous map $\widetilde T:X\to Y$ satisfying
		\begin{eqnarray*}
		T(x)-\widetilde T(x)\in U_0
		\quad\text{for every }x\in K.
	\end{eqnarray*}
		If the theorem is proved for $\widetilde T$ with approximation
		accuracy $U_0$, the resulting approximation differs from the original
		map $T$ by an element of $U_0+U_0\subseteq U$ on $K$. 
		
		We now prove the result of a uniformly continuous $T$. Given a fixed $U$, let $V \in \mathcal U$ be such that $V + V + V \subseteq U$. By uniform continuity of $T$, let $V' \in \mathcal U$ be a convex open such that $T(z_1) - T(z_2) \in V$ if $z_1-z_2\in V'$. We cover the compact $K$ by the open sets $x+V'$, and find a finite subcovering $\mathcal A = \{x_1+V', \ldots, x_{n'} +V' \}$. Let $\{\eta_i\}_{i=1}^{n'} \cup \{\eta_\infty\}$ be a partition of unity subordinate to $\mathcal A \cup \{X-K\}$, and define the map $P_X(x) = \sum_{i=1}^{n'}\eta_i(x)x_i$, where $\eta_i$ is subordinate to $x_i + V'$ and $\eta_\infty$ is subordinate to $X-K$. This map $P_X$ is well defined and continuous on $X$. Then, for each $x\in K$ we have
		\begin{equation*}
		x - P_X(x) = x - \sum_{i=1}^{n'}\eta_i(x)x_i = \sum_{i=1}^{n'}\eta_i(x)x - \sum_{i=1}^{n'}\eta_i(x)x_i = \sum_{i=1}^{n'}\eta_i(x)(x-x_i) \in V', 
		\end{equation*}
		since $V'$ is convex, and whenever $\eta_i(x) \neq 0$ we have that $x-x_i\in V'$. Let $E_X$ denote the vector subspace of $X$ spanned by $x_1, \ldots, x_{n'}$, and let $n := \dim E_X$. We choose a linear isomorphism $\varphi_n : E_X \longrightarrow \R^n$, which by Theorem~1.21 in \cite{Rudin} is a homeomorphism. Set $L = TP_X(K)$, which is compact in $Y$ by continuity of $T$ and $P_X$. As for the construction of $P_X$, we can obtain a continuous map $P_Y$, such that $z - P_Y(z) \in V$, whenever $z \in L$. We let $E_Y$ denote the finite dimensional subspace of $Y$ spanned by the image of $P_Y$, and let $m := \dim E_Y$. 
		We fix a linear homeomorphism $\varphi_m : E_Y \longrightarrow \R^m$. We define the map $T_{n,m} : P_X(K) \longrightarrow E_Y$ as $T_{n,m} := P_Y\circ T_{|P_X(K)}$. Let $F_{n,m} := \varphi_m\circ T_{n,m} \circ \varphi_n^{-1} : \varphi_n(P_X(K)) \subset \R^n \longrightarrow \R^m$. Let $\varepsilon > 0$ be such that $\varphi^{-1}_m(w_1) - \varphi^{-1}_m(w_2) \in V$ whenever $\|w_1 - w_2\| < \varepsilon$, where $\|\cdot \|$ indicates the Euclidean norm of $\R^m$. 
		Since $\varphi_nP_X(K)$ is compact, we can approximate $F_{n,m}$ by a neural network $f_{n,m}$ in the uniform norm of $C(\varphi_nP_X(K),\R^m)$ with accuracy $\varepsilon$, by the classical Universal Approximation Theorem for neural networks \cite{Pinkus}. In other words, we have $\|F_{n,m}(u) - f_{n,m}(u)\| < \varepsilon$ for all $u\in \varphi_nP_X(K)$. Then, we have
		\begin{eqnarray*}
			\lefteqn{T(x) - \varphi^{-1}_mf_{n,m}\varphi_nP_X(x)}\\
			& = & T(x) - TP_X(x) + TP_X(x) - P_YTP_X(x) + P_YTP_X(x) - \varphi_m^{-1}f_{n,m}\varphi_nP_X(x)\\
			& \in &  V + V + V \subseteq U,  
		\end{eqnarray*}
		where each difference is seen to be in $V$ as follows. Since $x-P_X(x) \in V'$ for each $x\in K$, by choice of $V'$ in terms of the uniform continuity of $T$, hence we have that $T(x) - TP_X(x)\in V$. By construction of $P_Y$, we have that $y-P_Y(y) \in V$ whenever $y\in L = TP_X(K)$, and therefore $TP_X(x) - P_YTP_X(x)\in V$. Lastly, since $TP_X(x) = T_{|P_X(K)}P_X(x)$ for each $x\in K$, and $P_YT_{|P_X(K)} = T_{n,m}$, we have that $P_YTP_X(x) = T_{n,m}P_X(x)$ for each $x\in K$. Moreover, $T_{n,m} = \varphi^{-1}_mF_{n,m}\varphi_n$ on $P_X(K)$ by definition of $F_{n,m}$. Hence, $P_YTP_X(x) = \varphi^{-1}_mF_{n,m}\varphi_n(P_X(x))$. Since $f_{n,m}$ approximates $F_{n,m}$ up to precision $\varepsilon$ on $\varphi_n(P_X(K))$, we have that $\|F_{n,m}\varphi_n(P_X(x)) - f_{n,m}\varphi_n(P_X(x))\| < \varepsilon$. Therefore, by choice of $\varepsilon$, we have that $P_YTP_X(x) - \varphi_m^{-1}f_{n,m}\varphi_nP_n(x) = \varphi^{-1}_mF_{n,m}\varphi_n(P_X(x)) - \varphi_m^{-1}f_{n,m}\varphi_nP_n(x) \in V$. This completes the proof. 
	\end{proof}

	We point out some interesting facts related to the proof of Theorem~\ref{thm:Universal}. The first observation is that the use of neural networks is related only to the approximation of the map $F_{n,m} : \varphi(P_X(K)) \subset \R^n \longrightarrow \R^m$ in the uniform norm. Therefore, any other universal approximator of real vector functions in the uniform Euclidean norm on compacts, could replace the neural network. In other words, the approach of the proof does not depend on neural networks in a strict sense. Our main interest in this class of approximators relates mostly to their recent vast use in deep learning. 
	
	Secondly, we note that while the approximation error guarantees hold only for vectors $x\in K$, the resulting map $\varphi_m^{-1}f_{n,m}\varphi_nP_X$ is actually well defined and continuous on $X$ because the neural network $f_{n,m}$ is well defined on $\R^n$, for a standard feed-forward neural network, even though its approximation guarantees hold only over the compact $\varphi_n(P_X(K))$.  
	
	 Our objective now is to consider how to concretely determine partitions of unity for the construction of the maps $P_X$ and $P_Y$ in the proof of Theorem~\ref{thm:Universal}, with the aim of removing the paracompactness hypothesis.  
	 In \cite{Projection}, the map $P_m$|which corresponds to the present $P_Y$|can be automatically extended from $L = P_X(K)$ to $Y$, and therefore we can consider the map $T_{n,m}$ defined on $\R^n$ rather than the restriction to $P_X(K)$. This extension is not automatic in the present work without paracompactness, since for example we cannot directly use the Dugundji extension theorem. However, following the same procedure as in the first part of the proof above, applying Kat{\v{e}}tov's Theorem 3 in \cite{Katetov}, we can show that the map $P_Y$ can be extended from $L$ to $Y$ as a uniformly continuous map $P_Y : Y \longrightarrow E_m$. After the extension it is unclear what $P_Y$ is outside of $L$, and therefore there seems to be no clear conceptual advantage in either perspective. We have therefore settled for the conceptually simpler domain restriction.

	The following result uses standard arguments, and it is a reformulation of the fact that compact sets are paracompact, along with the characterization of Hausdorff paracompact spaces via the existence of partitions of unity subordinate to open coverings. We provide an explicit proof as it shows how to construct the partitions by means of the seminorms. As such, in combination with neural networks it provides a formula-based expression of the approximators, along with their existence.  
	
	\begin{lemma}\label{lem:no_paracompact}
		Let $K$ be a compact subset of a locally convex TVS $X$, and let $\mathcal A = \{A_i\}_{i=1}^m$ be a covering of $K$, where each $A_i$ is relatively open in $K$. Then, we can construct a partition of unity $\{\eta_i\}_{i=1}^n$ subordinate to $\mathcal A$.  
	\end{lemma}
	\begin{proof}
		We let $\{p_\lambda\}_{\lambda\in \Lambda}$ denote a family of seminorms that induce the topology of $X$. 
		For each $i= 1,\ldots, m$, let $O_i$ be open in $X$ such that $A_i = O_i\cap K$. For each $x\in K$, choose $j(x) \in \{1, \ldots, m\}$ such that $x\in O_{j(x)}$. Since $O_{j(x)}$ is open in $X$, we can find finitely many seminorms $p_{\lambda_1}, \ldots, p_{\lambda_{n(x)}}$ and $\varepsilon_x > 0$ such that $y \in O_{j(x)}$ for each $y\in \bigcap_{r=1}^{n(x)} B_{\lambda_r}(x,2\varepsilon_x)$, where $B_{\lambda_r}(x,d)$ denotes the ball of radius $d$ centered in $x$ with respect to the seminorm $p_{\lambda_r}$. We define the functions $q_x(y) := \max_{1\leq r \leq n(x)} p_{\lambda_r}(y)$, and $h_x(y) := \max \{0, \varepsilon_x - q_x(y-x)\}$. By construction,  we have $\supp h_x \subset O_{j(x)}$. We cover the compact $K$ by the open sets $B_{q_x}(x,\epsilon_x)$, with $x\in K$, where $B_{q_x}(x,\varepsilon_x)$ is the ball of radius $\varepsilon_x$ in the $q_x$ seminorm defined above. By compactness, we can find $x_1, \ldots, x_n$ such that $K \subset \bigcup_{i=1}^n B_i(x_i,\varepsilon_i)$, where $B_i(x_i,\varepsilon_i)$ is the unit ball in the $q_{x_i}$ seminorm with radius $\varepsilon_{x_i}$ centered at $x_i$. We set $H := \sum_{i=1}^n h_{x_i}$, and notice that by construction $H(x) \neq 0$ for every $x\in K$. Therefore, setting
		\begin{eqnarray*}
			\eta_i(x) &=& \frac{h_i(x)}{H(x)}, 
		\end{eqnarray*}
		we see that $\{\eta_i\}_{i=1}^n$ is a partition of unity subordinate to the covering $\mathcal A$. 
	\end{proof}
	
	The paracompactness hypothesis can therefore be removed from Theorem~\ref{thm:Universal}, since its use in the proof related only to the construction of a partition of unity subordinate to a given finite open covering of a compact. By Lemma~\ref{lem:no_paracompact} we can construct such partitions even when the underlying space is not paracompact, and the proof of Theorem~\ref{thm:Universal} would remain unchanged otherwise. We therefore have the following result. 
	
	\begin{theorem}\label{thm:Universal_no-paracompact}
		Let $X$ and $Y$ be locally convex TVSs, let $T: X\longrightarrow Y$ be a continuous map, and let $K\subset X$ be a compact subset. Then, for any choice of open neighborhood $U$ of $0$, there exist natural numbers $n,m\in \mathbb N$, finite dimensional subspaces $E_n \subset X$ and $E_m \subset Y$, a continuous map $P_X : K\longrightarrow E_n$ (defined and continuous on a neighborhood of $K$), and a neural network $f_{n,m} : \mathbb R^n \longrightarrow \mathbb R^m$ such that for every $x\in K$
		\begin{eqnarray}
		T(x) - \varphi_m^{-1}f_{n,m}\varphi_nP_X(x) &\in& U,
		\end{eqnarray}
		where $\varphi_k : E_k \longrightarrow \mathbb R^k$ indicates a linear homeomorphism between the finite dimensional space $E_k$ and $\mathbb R^k$.  
	\end{theorem}

	Applying \cite{Katetov} one could obtain $P_X$ well defined on the whole space $X$. However, we would not know how $P_X$ is defined outside $K$. However, we notice that $P_X$ is well defined and continuous on a neighborhood of $K$ as a consequence of the proof of Lemma~\ref{lem:no_paracompact}, where each function $\eta_i$ is well defined and continuous on the open set $\bigcup_{i=1}^n B_i(x_i,\varepsilon_i)$ containing $K$. Each $\eta_i$ can be explicitly defined in terms of finitely many seminorms of $X$, and therefore they could also be implemented on a computer.

	\section{Universal approximation in non-locally convex topological vector spaces}
	
	When the spaces are not locally convex, we can obtain results similar to the previous section over subsets that have finite topological dimension. 
	
	\begin{lemma}\label{lem:nonlocal_T_approx}
				Let $X$ and $Y$ be paracompact TVSs, let $T: X\longrightarrow Y$ be a continuous map, and let $K$ be a compact subset of $X$ with finite topological dimension. Then, for any open neighborhood $U$ of $0$, we can construct a continuous map $F: K \longrightarrow Y$ with finite dimensional range $E$, which approximates $T$ up to $U$ precision on $K$, i.e. such that $T(x) - F(x) \in U$ for all $x\in K$. Such $F$ can be extended to $X$ as a uniformly continuous function $\hat F : X \longrightarrow E$. 
	\end{lemma}
	\begin{proof}
			Let $d = \dim K$ denote the topological dimension of $K$. Therefore, any covering of $K$ has a refinement of order at most $d+1$. 
			Let $W \in \mathcal U$ such that $W + \cdots + W \subset U$, where the sum is taken $d+1$ times. Consider the open covering of $K$ given by $\{T^{-1}(T(x) + W)\}_{x\in K}$. Therefore, by construction, if $y\in T^{-1}(T(x) + W)$, we have that $T(y) - T(x) \in W$. Let $\{A_\lambda\}_{\lambda\in \Lambda}$ be a refinement by relative opens of order at most $d+1$, where $\Lambda$ is an indexing set. This means that each $A_\lambda$ is contained inside $T^{-1}(T(x) + W)$ for some $x\in K$. By compactness, we can find finitely many $\lambda$'s, $\lambda_1, \ldots, \lambda_r$ such that $K \subset \bigcup_{i=1}^r A_{\lambda_i}$. We let $O_{\lambda_i}$ denote opens in $X$ such that $K \cap O_{\lambda_i} = A_{\lambda_i}$ for each $i$. Since each $A_{\lambda_i}$ is contained in $T^{-1}(T(x_i) + W)$ for some $x_i\in K$, we find $x_1, \ldots, x_k\in K$ such that each $T(A_{\lambda_i})$ is contained in one of the opens $T(x_j) + W$ where, possibly allowing repetitions among the $x_j$'s, we can assume that $k = r$, and each $T(A_{\lambda_i}) \subset T(x_i) + W$. 
			Let $\{\eta_i\}_{i=1}^r \cup \{\eta_\infty\}$ be a partition of unity subordinate to the covering $\{O_{\lambda_i}\}_{i=1}^r \cup \{X-K\}$, where $\supp \eta_i \subset O_{\lambda_i}$ and $\supp \eta_\infty \subset X-K$. We define the continuous map $F : K \longrightarrow Y$ by
			\begin{eqnarray*}
				F(x) &=& \sum_{i=1}^r \eta_i(x)T(x_i), 
			\end{eqnarray*}
			which is uniformly continuous on $K$. 
			We show that $F$ approximates $T$ with $U$-precision on the whole compact $K$. For any $x\in K$, we have that $x \in A_{\lambda_{j(t)}}$ for $t = 1, \ldots, q$ and $j(t) \in \{1,\ldots,k \}$, where $q\leq d+1$ since the order of $\{A_\lambda\}$ is $d+1$. Therefore $\eta_{\lambda_i}(x) \neq 0$ for at most $d+1$ terms. It follows that
			\begin{eqnarray*}
				T(x) - F(x) = \sum_j\eta_{\lambda_j}(x)(T(x) - T(x_j)) \in \overbrace{W + \cdots + W}^{{\rm (d+1)-times}} \subset U.
			\end{eqnarray*} 
			By Kat{\v{e}}tov's theorem \cite{Katetov}, we can extend each restriction ${\eta_i}_{|K}$ to a uniformly continuous function over the whole space $X$, and we will denote such extension by $\hat \eta_i$. Then, we set 
			\begin{eqnarray*}
				\hat F(x) &=& \sum_{i=1}^r \hat \eta_i(x)T(x_i), 
			\end{eqnarray*}
			which is well defined and uniformly continuous on the whole space $X$. Such $\hat F$ is the claimed uniformly continuous extension. 
	\end{proof}

	In the following, neither $X$ nor $Y$ need to be locally-convex. 
	\begin{theorem}\label{thm:Universal_NLC}
		Let $X$ and $Y$ be TVSs, where $X$ is paracompact, and let $T: X\longrightarrow Y$ be a continuous map. Suppose $K\subset X$ is a compact subset with finite topological dimension. Then, for any choice of an neighborhood $U$ of $0$, there exist natural numbers $n,m\in \mathbb N$, finite dimensional subspaces $E_n \subset X$ and $E_m \subset Y$, a continuous map $P_X : X\longrightarrow E_n$, and a neural network $f_{n,m} : \mathbb R^n \longrightarrow \mathbb R^m$ such that for every $x\in K$
		\begin{eqnarray}
		T(x) - \varphi_m^{-1}f_{n,m}\varphi_nP_X(x) &\in& U,
		\end{eqnarray}
		where $\varphi_k : E_k \longrightarrow \mathbb R^k$ indicates a linear homeomorphism between the finite dimensional space $E_k$ and $\mathbb R^k$.  
	\end{theorem}
	\begin{proof}
			We show how to apply Lemma~\ref{lem:nonlocal_T_approx} to replace the use of local convexity in the proof of Theorem~\ref{thm:Universal}. 
			
			Let $V\in \mathcal U_Y$ be such that $V + V + V \subseteq U$. Applying Lemma~\ref{lem:nonlocal_T_approx} to $T$ on $K$, we find a uniformly continuous map $\tilde T : X \longrightarrow E_m$, where $E_m \leq Y$ is a subspace of $Y$ with dimension $m$, such that $T(x) - \tilde T(x) \in V$ for each $x\in K$. We let $V'\in \mathcal U_X$ be such that $\tilde T(z_1) - \tilde T(z_2) \in V$ whenever $z_1-z_2\in V'$, by uniform continuity of $\tilde T$. 
			We can construct a continuous map $P_X : X\longrightarrow E_n$ such that $x-P_X(x) \in V'$ for each $x\in K$, by applying Lemma~\ref{lem:nonlocal_T_approx} to the identity map on $X$, where $n$ is the dimension of $E_n \leq X$. Let $C = \varphi_n(P_X(K)) \subset \R^n$ be compact, and define $G: C \longrightarrow \R^m$ as $G = \varphi_m\tilde T{\varphi_n^{-1}}_{|C}$. Let $f_{n,m}$ be a neural network approximating $G$ with sufficiently high precision in the uniform norm of $C(C,\R^m)$ induced by the Euclidean norm of $\R^m$ as in the proof of Theorem~\ref{thm:Universal}. Namely, we have $\|G(u) - f_{n,m}(u)\| < \varepsilon$ for each $u\in C$, where $\varepsilon$ is chosen in such a way that $\varphi_m^{-1}(w_1) - \varphi_m^{-1}(w_2) \in V$ whenever $\|w_1-w_2\| < \varepsilon$, and $\|\cdot\|$ indicates the Euclidean norm of $\R^m$. Then, following the same procedure of Theorem~\ref{thm:Universal} we see that for each $x\in K$
			\begin{eqnarray*}
					&&T(x) - \varphi_m^{-1}f_{n,m}\varphi_nP_X(x) = T(x) - \tilde T(x) + \tilde T(x) - \tilde T(P_X(x)) + \tilde T(P_X(x)) - \varphi_m^{-1}f_{n,m}\varphi_nP_X(x)\\
					&& \in V + V + V \subseteq U. 
			\end{eqnarray*}
	\end{proof}


\begin{thebibliography}{0}
		
		\bib{Bil-Xan}{article}{
			title={A universal approximation theorem and its applications to vector lattice theory},
			author={Bilokopytov, Eugene},
			author={Xanthos, Foivos},
			journal={Journal of Mathematical Analysis and Applications},
			pages={130632},
			year={2026},
			publisher={Elsevier}
		}
		
		\bib{Chen-Chen}{article}{
			title={Universal approximation to nonlinear operators by neural networks with arbitrary activation functions and its application to dynamical systems},
			author={Chen, Tianping},
			author={Chen, Hong},
			journal={IEEE transactions on neural networks},
			volume={6},
			number={4},
			pages={911--917},
			year={1995},
			publisher={IEEE}
		}
		
		\bib{Frolik}{article}{
			title={Existence of $\ell_\infty$-partitions of unity},
			author={Frol{\'\i}k, Zden{\v{e}}k},
			journal={Rend. Sem. Mat. Univ. Politec. Torino},
			volume={42},
			number={1},
			pages={9--14},
			year={1984}
		}
		
		\bib{Gar-Mil}{article}{
			title={On the Extension of Uniformly Continuous Functions (1)},
			author={Gardner, LT},
			author={Milnes, P},
			journal={Canadian Mathematical Bulletin},
			volume={18},
			number={1},
			pages={143--145},
			year={1975},
			publisher={Cambridge University Press}
		}
		
		\bib{Ismailov}{article}{
			title={Universal approximation theorem for neural networks with inputs from a topological vector space},
			author={Ismailov, Vugar E},
			journal={Information Processing Letters},
			pages={106623},
			year={2026},
			publisher={Elsevier}
		}
	
		\bib{Kantorovich-Akilov}{book}{
			title={Functional analysis},
			author={Kantorovich, Leonid Vitalevich},
			author={Akilov, Gleb Pavlovich},
			year={2016},
			publisher={Elsevier}
		}
		
		\bib{Katetov}{article}{
			title={On real-valued functions in topological spaces},
			author={Kat{\v{e}}tov, Miroslav},
			journal={Fundamenta Mathematicae},
			volume={38},
			pages={85--91},
			year={1951},
			publisher={Instytut Matematyczny Polskiej Akademii Nauk}
		}
	
		\bib{Katetov2}{article}{
			title={Correction to" On real-valued functions in topological spaces"},
			author={Kat{\v{e}}tov, Miroslav},
			journal={Fundamenta Mathematicae},
			volume={40},
			pages={203--205},
			year={1953},
			publisher={Instytut Matematyczny Polskiej Akademii Nauk}
		}
	
		\bib{Kovachki}{article}{
			title={Operator learning: Algorithms and analysis},
			author={Kovachki, Nikola B},
			author={Lanthaler, Samuel},
			author={Stuart, Andrew M},
			journal={Handbook of Numerical Analysis},
			volume={25},
			pages={419--467},
			year={2024},
			publisher={Elsevier}
		}
	
		\bib{Lanthaler}{article}{
			title={Error estimates for deeponets: A deep learning framework in infinite dimensions},
			author={Lanthaler, Samuel},
			author={Mishra, Siddhartha},
			author={Karniadakis, George E},
			journal={Transactions of Mathematics and Its Applications},
			volume={6},
			number={1},
			pages={tnac001},
			year={2022},
			publisher={Oxford University Press}
		}
	
		\bib{Lu}{article}{
			title={Learning nonlinear operators via DeepONet based on the universal approximation theorem of operators},
			author={Lu, Lu},
			author={Jin, Pengzhan},
			author={Pang, Guofei},
			author={Zhang, Zhongqiang},
			author={Karniadakis, George Em},
			journal={Nature machine intelligence},
			volume={3},
			number={3},
			pages={218--229},
			year={2021},
			publisher={Nature Publishing Group UK London}
		}
		
		\bib{Pinkus}{article}{
			title={Approximation theory of the MLP model in neural networks},
			author={Pinkus, Allan},
			journal={Acta numerica},
			volume={8},
			pages={143--195},
			year={1999},
			publisher={Cambridge University Press}
		}
		
		\bib{Rob-Rob}{article}{
			title={Uniformly Continuous Partitions of Unity on a Metric Space},
			author={Robinson, Stewart M},
			author={Robinson, Zachary},
			journal={Canadian Mathematical Bulletin},
			volume={26},
			number={1},
			pages={115--117},
			year={1983},
			publisher={Cambridge University Press}
		}
		
		\bib{Rudin}{book}{
			title={Functional Analysis},
			author={Rudin, Walter},
			edition={2nd},
			year={1991},
			publisher={McGraw-Hill Education},
		}
	
		\bib{Saini}{article}{
			title={A Universal Approximation Theorem for Neural Networks with Outputs in Locally Convex Spaces},
			author={Saini, Sachin},
			journal={arXiv preprint arXiv:2603.07242},
			year={2026}
		}
		
		\bib{Projection}{article}{
			title={Projection Methods for Operator Learning and Universal Approximation},
			author={Zappala, Emanuele},
			journal={arXiv preprint arXiv:2406.12264},
			year={2024}
		}
		
		\bib{Leray-Schauder}{article}{
			title={Leray-Schauder Mappings for Operator Learning},
			author={Zappala, Emanuele},
			journal={arXiv preprint arXiv:2410.01746},
			year={2024}
		}
		
	\end{thebibliography}
\end{document}